\documentclass[12pt,cd]{amsart}
\usepackage{latexsym}
\usepackage{amsmath,amsfonts,amssymb,amsthm,mathtools,amscd,bm,mathabx}
\usepackage{graphics}
\usepackage{csquotes}
\usepackage{blindtext}
\usepackage{xcolor}
 2

\usepackage{hyperref}
\newtheorem{theorem}{Theorem}[section]
\newtheorem{prop}[theorem]{Proposition}

\newtheorem{definition}[theorem]{Definition}

\newtheorem{Note}[theorem]{Note}
\newtheorem{remark}[theorem]{Remark}
\numberwithin{equation}{section}

\title{Harmonic Enneper Immersion in $\mathbb{R}^3$}

\author{Priyank Vasu}

\address{Department of Mathematics, Indian Institute of Technology Patna, Bihta, Patna-801106, Bihar, India}

\thanks{Priyank Vasu: \href{mailto:priyank_2121ma16@iitp.ac.in}{priyank\_2121ma16@iitp.ac.in}}

\begin{document}

\maketitle

\begin{abstract}

We present a method for constructing harmonic immersions in $\mathbb{R}^3$, known as the Enneper-type representation. We also prove that any harmonic immersion in $\mathbb{R}^3$ can be obtained using this approach. Furthermore, we determine the number of non-planar rotational harmonic immersions in $\mathbb{R}^3$ that connect two coaxial circles in parallel planes, where both circles have the same radius $r > 0$ and are separated by a distance $l > 0$.

\end{abstract}

\section{Introduction}

An immersed surface in $\mathbb{R}^3$  is called \textit{harmonic immersion} if it admits a parametrization whose coordinate functions are harmonic (the Dirichlet equation $\Delta X = 0$, where $\Delta$ denotes the Laplace operator, defined by $\Delta = \partial_{xx} + \partial_{yy}$ in local coordinates $(x, y)$). In particular, the theory of minimal surfaces is closely related to the study of conformal harmonic immersions. Among the most well-known harmonic immersions are minimal surfaces, where the Weierstrass representation provides a harmonic parametrization.  It is well known that a conformal immersion of a Riemann surface into Euclidean 3-space is minimal if and only if it is harmonic, and in this case, its Gauss map is conformal. However, harmonic immersions are not necessarily conformal (for example, $Y_1$ in \eqref{example}).
 The relationship among conformality, minimality, and harmonicity plays a central role in geometric analysis, especially within the study of surfaces. Our goal in this paper is to analyze the case when a surface is simply a harmonic immersion without necessarily being minimal.

The study of harmonic immersions in $\mathbb{R}^3$ began in the 1960s with the work of Klotz \cite{Klotz1967,Klotz1968,KlotzMilnor1979,KlotzMilnor1980}, and around the same time, Osserman \cite{Osserman1986} developed the foundations of the global theory of minimal surfaces. Klotz focused on distinguishing properties unique to minimal surfaces from those that apply more generally to harmonic immersions. More recently, researchers have extended this work to study harmonic quasiconformal immersions in $\mathbb{R}^3$, properties of the Gauss map, and Weierstrass-type formulas for harmonic immersions \cite{AlarconLopez2013,Kalaj2013}. Additionally, the Gauss-Bonnet formula for harmonic immersions and the global study of harmonic immersions of arbitrary genus have been investigated in \cite{Connor2015,Connor2018}.

In this paper, we introduce a new type of representation for harmonic immersions, inspired by Andrade's method \cite{Andrade1998}, which provides an alternative way to construct minimal surfaces in Euclidean 3-space. This method is equivalent to the classical Weierstrass representation and is commonly referred to as the Enneper representation. As an application, we address the problem of determining the number of non-planar rotational harmonic immersions that connect two coaxial circles in parallel planes. Both circles have the same radius $r > 0$ and are separated by a distance $l > 0$.

The paper is organized as follows. In Section \ref{Preliminaries}, we recall some fundamental concepts of harmonic immersions. In Section \ref{main results}, we introduce an Enneper-type representation for harmonic immersions and show that any harmonic immersion can be locally expressed as the \textit{harmonic Enneper graph} of a real-valued harmonic function with an associated Hopf differential defined on a complex domain. In Section \ref{application}, we classify all rotational harmonic immersions and apply these results to determine the number of rotational harmonic immersions that connect two coaxial circles of the same radius in parallel planes.

	\section{Preliminaries} \label{Preliminaries}
    Let $\Omega$ be a connected open set in the complex plane $\mathbb{C}$. We 
identify $\mathbb{R}^2$ with $\mathbb{C}$ by writing $z = x + iy$, where $ (x, y)\in \mathbb{R}^2$. 
\begin{definition}
     A map $X = (X_j )_{j=1,2,3}: \Omega \rightarrow \mathbb{R}^3$ is said to be a harmonic
immersion if $X$ is an immersion and $X_j;\, j=1,2,3,$ are harmonic functions on $\Omega.$
A subset $\mathcal{S} \subset \mathbb{R}^3$ is said to be a harmonic surface if there exists a harmonic immersion $X: \Omega  \rightarrow \mathbb{R}^3$ such that $\mathcal{S} = X(\Omega)$. In this case, $X$ is said to be a
harmonic parameterization of $\mathcal{S}$.

\end{definition}
A harmonic immersion may admit different harmonic parameterizations. For instance,
\begin{equation} \label{example}
    \begin{aligned}
        Y_1 &: \mathbb{C} \to \mathbb{R}^3, \quad &Y_1(z) &= \operatorname{Re} \big(e^z, z e^z, iz \big), \\
        Y_2 &: \{\operatorname{Re}(z) > 0\} \to \mathbb{R}^3, \quad &Y_2(z) &= \operatorname{Re} \big( \sinh(z), i\cosh(z), iz \big),
    \end{aligned}
\end{equation}
are two different harmonic parameterizations of the same surface (the half helicoid).

We define the Wirtinger operators $\partial_z$ and $\partial_{\bar{z}}$ as the complex differential operators given by
$\partial_z := \frac{\partial}{\partial x} - i \frac{\partial}{\partial y}, \quad \partial_{\bar{z}} := \frac{\partial}{\partial x} + i \frac{\partial}{\partial y},$
where \( z = x + i y \in \Omega \). Let \( X = (X_j)_{j=1,2,3} \) be a harmonic immersion. Also, we define
$$
\Phi=\left(\phi_j\right)_{j=1,2,3}:=\left(\partial_z X_j\right)_{j=1,2,3} .
$$
\begin{definition}\label{definition Weierstrass representation} The couple $(\Omega, \Phi)$ is called the Weierstrass representation of $X$. The holomorphic 2-form $\mathfrak{H}:=\sum_{j=1}^3 \phi_j^2$ on $\Omega$ is said to be the Hopf differential of the harmonic immersion $X$.\end{definition} 

Let $z = x + iy \in \Omega$, and write $X_x = \frac{\partial X}{\partial x} ,X_y = \frac{\partial X}{\partial y}, X_z=\partial_z X, X_{\bar{z}}=\partial_{\bar{z}} X  $. Define $||\Phi||:= \left(\sum _{j=1}^3|\phi_j|^2\right)^{1/2}$. Then it follows that $||X_x \times X_y||=\sqrt{||\Phi||^2-|\mathfrak{H}|^2}$.
%\begin{lemma}[*]\label{harmonic immersion}
   % Let $\Phi=\left(\Phi_j\right)_{j=1,2,3}$ be a triple of holomorphic 1-forms on $M$ such that
    %\begin{enumerate}
      %  \item $\left|\sum_{j=1}^3 \Phi_j^2\right|<\|\Phi\|^2$ everywhere on $M$,
     %   \item $\Phi$ has no real periods on $M$.
    %\end{enumerate} 

%Then the map $X: M \rightarrow \mathbb{R}^3, X(P)=\Re\left(\int^P \Phi\right)$, is a harmonic immersion with Hopf differential $\mathfrak{H}=\sum_{j=1}^3 \Phi_j^2$ on $M$.
%\end{lemma}

\begin{theorem}\cite{AlarconLopez2013}\label{harmonic immersion}
   Let $X: \Omega \rightarrow \mathbb{R}^3$ be a harmonic immersion. Then, 
   \begin{equation*}
       \Phi=\left(\phi_j\right)_{j=1,2,3}=\left(\partial_z X_j\right)_{j=1,2,3} ;
   \end{equation*}
 satisfies the following conditions:
    \begin{enumerate}
   % \item $\sum_{j=1}^3 \Phi_j^2=\mathfrak{H}$
        \item $\left|\sum_{j=1}^3 \phi_j^2\right|<\|\Phi\|^2$ everywhere on $\Omega$,
        \item ${\partial_{\bar{z}} \phi_j}=0 \quad for \; j=1,2,3$.
    \end{enumerate} 
   Conversely, if $\Omega $ is a simply connected domain and $\phi_j;\, j=1,2,3,$ are holomorphic functions satisfying the conditions above, then the map 
   $$X=\operatorname{Re}\left(\int^z_{z_{0}} \Phi \;dz\right),$$
   is a well-defined harmonic immersion (here,
$z_{0}$ is an arbitrary fixed point of $\Omega$).
\end{theorem}
\begin{Note}
    The first condition in Theorem \ref{harmonic immersion} guarantees that $X$ is an immersion. Moreover, if the Hopf differential satisfies $\mathfrak{H} \equiv 0$ in $\Omega$, then the corresponding harmonic immersion $X$ is minimal.
\end{Note}   
Let \( X: \Omega \to \mathbb{R}^3 \) be a regular parametrization of the smooth surface \( \mathcal{S} = X(\Omega) \).
Also, let \( S^2 \) be the unit sphere in \( \mathbb{R}^3 \). Then the orientation-preserving Gauss map \( {n}: \Omega \to S^2 \) of \( \mathcal{S} \) is defined as follows:
\[
{n} = \frac{X_x \times X_y}{||X_x \times X_y||}.
\]
If we denote by $\xi: S^2-\{(0,0,1)\} \rightarrow \mathbb{R}^2$, $\xi\left(x_1, x_2, x_3\right)=\left(x_1 /\left(1-x_3\right), x_2 /\left(1-x_3\right)\right)$ the stereographic projection, then the orientation-preserving map $g:=\xi \circ n$ is said to be the complex Gauss map of $X$.

 An orientation-preserving smooth mapping \( k: \Omega \to \Omega' \) between two domains in \( \mathbb{C} \) is called \textit{quasiconformal} if \( |\mu| < 1 \), where \( \mu \) is the \textit{Beltrami differential} (first complex
dilatation) of \( k \), given by $
\mu := \frac{ k_{\bar{z}}}{ k_z}$.

\begin{definition}[\cite{Kalaj}]
A smooth mapping \( X : \Omega \to \mathbb{R}^3 \) is called quasiconformal if, 
\[
D_X := \frac{||X_x||^2 + ||X_y||^2}{2\|X_x \times X_y\|} \leq 1,
\]
on $\Omega.$
\end{definition}

Since the stereographic projection is conformal, the Gauss map \( n \) is quasiconformal if and only if the map \( g \) is quasiconformal.

\begin{definition}\cite{AlarconLopez2013}
    A harmonic immersion $X: \Omega \rightarrow \mathbb{R}^3$ is said to be quasiconformal if its orientation-preserving Gauss map $n: \Omega \rightarrow S^2$ is quasiconformal (or, equivalently, if $g$ is quasiconformal). In this case, $X$ is called a harmonic quasiconformal parameterization of the harmonic surface $\mathcal{S}=X(\Omega)$.  
\end{definition}

\begin{Note}
    Notice that \( Y_2 \) is conformal, whereas \( Y_1 \) is quasiconformal in equation \eqref{example}.
\end{Note}

\section{Main results}\label{main results}
In this section, we prove an Enneper-type representation formula for harmonic immersion.

\begin{theorem}\label{theorem 1}
    Let $h: \Omega \rightarrow \mathbb{R}$ be a harmonic function and $\mathfrak{H}$ a holomorphic 2-form in the simply connected domain $\Omega$. Suppose that
 $L,P : \Omega \rightarrow \mathbb{C}$ are two holomorphic functions such that the following conditions are
satisfied:

\begin{equation}\label{condition 1 in theorem 1}
    L_zP_z+(h_z)^2=\mathfrak{H}
\end{equation}
and
\begin{equation}\label{condition 2 in theorem 1}
    |L_z|\neq|P_z| \quad \text{on } \Omega,
\end{equation}
then the map $X:\Omega \rightarrow \mathbb{C}\times \mathbb{R}$ given by $X(z)= (L(z)+\overline{P(z)}, h(z))$ defines a harmonic immersion with Hopf differential $\mathfrak{H}$.
\end{theorem}
\begin{proof}
    Consider the three complex-valued functions on $\Omega$ given by:
    \begin{equation}
        \phi_1=\frac{L_z+P_z}{2},\;\phi_2=\frac{i(P_z-L_z)}{2},\; \phi_3=h_z.
    \end{equation}
    As $L_z=\phi_1+i\phi_2$ and $P_z=\phi_1-i\phi_2$, it follows from Theorem~\ref{harmonic immersion} that
    \begin{equation}\label{hopf diff equation}
        \sum_{j=1}^3 \phi_j^2=L_zP_z+(h_z)^2=\mathfrak{H}.
    \end{equation}
    Also, using \eqref{hopf diff equation}
   \begin{align*}
    2\left(\|\Phi\|^2 - \left|\sum_{j=1}^3 \phi_j^2\right|\right) 
    &= |L_z|^2 + |P_z|^2 + 2|h_z|^2 - 2|\mathfrak{H}| \\
    &= |L_z|^2 + |P_z|^2 + 2|h_z|^2 - 2|L_zP_z + (h_z)^2| \\
    &\geq |L_z|^2 + |P_z|^2 + 2|h_z|^2 - 2|L_zP_z| - 2|h_z|^2 \\
    &= (|L_z| - |P_z|)^2>0.
\end{align*}
We now observe that since $h$ is a harmonic function (i.e., 
$h_{xx} + h_{yy} = 0$), it follows that $\phi_3 = h_z$
is holomorphic. Moreover, the holomorphicity of $L$ and $P$ implies that the real
and imaginary parts of $L$ and $P$ are harmonic functions. In particular, we can write
\begin{equation*}
    \phi_1=\frac{\partial\operatorname{Re}(L+P)}{\partial z}, \quad \phi_2=\frac{\partial\operatorname{Im}(L-P)}{\partial z},
\end{equation*}
which implies that $(\phi_1)_{\bar{z}}=0=(\phi_2)_{\bar{z}}$. From Theorem \ref{harmonic immersion}, we conclude that
\begin{align*}
    X(z)&=2\left(\operatorname{Re}\int \phi_1(z) dz+i\operatorname{Re}\int \phi_2(z) dz,\operatorname{Re}\int \phi_3(z) dz\right)\\
    &=(L(z)+\overline{P(z)},h(z)),
\end{align*}
is a harmonic immersion with Hopf differential $\mathfrak{H}$.
\end{proof}

We call the immersion $X = (L + \bar{P}, h)$
a harmonic Enneper immersion associated with $h$, its image a \textit{harmonic Enneper graph} of $h$ with Hopf differential $\mathfrak{H}$ and
$D_X = (L_z, P_z, h_z)$ the \textit{Enneper data} of $X$.

Next, we show that any harmonic immersion can be rendered as the harmonic Enneper graph of some harmonic function.
\begin{theorem}
    Let $\Tilde{X} : \mathcal{M} \rightarrow \mathbb{R}^3 \equiv \mathbb{C} \times \mathbb{R}$ be a harmonic immersion from a Riemann surface
$\mathcal{M}$ in $\mathbb{R}^3$. Then, there exists a simply connected domain $\Omega$, a harmonic function
$h : \Omega  \rightarrow \mathbb{R}$, and a holomorphic function $\mathfrak{H} : \Omega  \rightarrow \mathbb{C}$ such that the immersed harmonic surface $\Tilde{X}(\mathcal{M})$ is a harmonic Enneper graph of $h$ with Hopf differential $\mathfrak{H}$.
\end{theorem}
\begin{proof}
    Suppose that the harmonic immersion is given by $\Tilde{X}=(\Tilde{X}_1+i\Tilde{X}_2,\Tilde{X}_3)$. Since $\Tilde{X}$ is a harmonic immersion, $\mathcal{M}$ cannot be compact ($\Tilde{X}$ would be a harmonic function on a compact Riemannian surface, hence constant). Now, by Koebe’s uniformization theorem, the universal covering space $\Omega$ of $\mathcal{M}$ is either the complex plane $\mathbb{C}$ or the open unit complex disc.
Let $\Pi: \Omega \rightarrow \mathcal{M}$ denotes the universal covering map, and define  $X: \mathcal{M} \rightarrow \mathbb{R}^3$ the lift of
$\Tilde{X}$ , i.e., $X = \Tilde{X} \circ \Pi$. Since $X$ is also a harmonic immersion, it follows that
$(X_1)_z,(X_2)_z$, and $(X_3)_z$ are holomorphic. Then, we have 
\begin{equation}\label{equation oh H}
    (X_1)_z^2+(X_2)_z^2+(X_3)_z^2= \mathfrak{H},
\end{equation}
$$
\mathfrak{H}=((X_1)_z+i (X_2)_z)((X_1)_z-i (X_2)_z)+(X_3)_z^2.
$$
Fix a point $z_0 \in \Omega$. Then, we can define
\begin{equation}
    L(z):=\int_{z_0}^z ((X_1)_z+i (X_2)_z)dz,\: P(z):=\int_{z_0}^z ((X_1)_z-i (X_2)_z)dz.
\end{equation}

Since $\Omega$ is a simply connected domain and the integrand functions are holomorphic, the above integrals do not depend on the path from $z_0$ to $z$, so $L$ and $P$ are well-defined holomorphic functions. We prove that $X(z)=(L(z)+\overline{P(z)}, h(z))$, where $h(z)=X_3(z)$ is a harmonic function (because $\left(X_3\right)_{z \bar{z}}=0$ ). For this, we note that

$$
\begin{aligned}
L(z)+\overline{P(z)} & =\int_{z_0}^z\left[\left(X_1\right)_z+i\left(X_2\right)_z\right] \mathrm{d} z+\int_{z_0}^z\left[\left(X_1\right)_{\bar{z}}+i\left(X_2\right)_{\bar{z}}\right] \mathrm{d} \bar{z} \\
& =\int_{z_0}^z \mathrm{~d} X_1+i \int_{z_0}^z \mathrm{~d} X_2=X_1(z)+i X_2(z).
\end{aligned}
$$
Where in the last equality, we have assumed (without loss of generality) that $X\left(z_0\right)=$ $(0,0,0)$. In addition, we observe that equation \eqref{equation oh H} can be written as
\begin{equation}\label{L_z P_z}
L_z P_z+\left(h_z\right)^2=\mathfrak{H}, 
\end{equation}
that is the condition \eqref{condition 1 in theorem 1} in Theorem \ref{theorem 1}. Finally, to prove that $X$ is an Enneper immersion associated to the harmonic function $h$, it remains to verify equation \eqref{condition 2 in theorem 1}. Further, we observe that
$$
\left(X_1\right)_z=\frac{L_z+P_z}{2}, \quad\left(X_2\right)_z=\frac{i\left(P_z-L_z\right)}{2},
$$
and taking into account \eqref{L_z P_z}, we have 
$$
0<2||X_x \times X_y||^2=\left|L_z\right|^2+\left|P_z\right|^2+2\left|h_z\right|^2-2|\mathfrak{H}|\leq\left(\left|L_z\right|-\left|P_z\right|\right)^2 ,
$$
since $X$ is an immersion. This completes the proof.
\end{proof}

Let \( f: \Omega \to \mathbb{C} \) be a complex harmonic function, where \( \Omega \) is a simply connected domain. Then \( f \) can be written as $f = L + \overline{P},$ where \( L \) and \( P \) are holomorphic functions in \( \Omega \). The quantity $v_f:= \frac{\bar{f}_{\bar{z}}}{ f_z}$, known as the
\textit{second complex dilatation}, turns out to be more relevant than the first complex
dilatation $\mu_f$. Since $|v_f |=|\mu_f |$, it follows that $f$ is quasiconformal if
and only if $|v_{f}(z)|<1$ on $\Omega$ (For details see \cite{Duren2004}).

Moreover, if \( f \) is a quasiconformal harmonic function, then the condition $|P_z| < |L_z|$ holds throughout \( \Omega \). Therefore, we can state the following:

\begin{prop}
    Let $f= L + \bar{P}: \Omega \rightarrow \mathbb{C}$ be a quasiconformal harmonic function on the simply connected domain $\Omega$. Suppose that
 $h : \Omega \rightarrow \mathbb{R}$ is a harmonic function.
Then the map $X:\Omega \rightarrow \mathbb{C}\times \mathbb{R}$ given by $X(z)= (L(z)+\overline{P(z)},h(z))$ defines a harmonic Enneper graph of $h$.
\end{prop}
We state the following remark, which follows from Proposition 2.12 in \cite{AlarconLopez2013}.
\begin{remark}
   A harmonic immersion $X$ is quasiconformal if and only if $\operatorname{sup}_{\Omega} \frac{|\mathfrak{H}|}{\|\Phi\|^2}<1$.
\end{remark}
As a consequence, we have the following:
\begin{prop}
    A harmonic immersion $X(z)=(L(z)+\overline{P(z)}, h(z))$ is quasiconformal if and only if
    $\operatorname{sup}_{\Omega} \frac{|L_z P_z+\left(h_z\right)^2|}{\left|L_z\right|^2+\left|P_z\right|^2+2\left|h_z\right|^2} <1$.
\end{prop}

We begin with the observation that if $\mathcal{D}_X=\left(L_z, P_z, h_z\right)$ is the Enneper data of a given harmonic Enneper immersion $X$ in $\mathbb{C} \times\mathbb{R}$, defined in the simply connected domain $\Omega$, and if $f: \Omega \rightarrow \mathbb{C}$ is a holomorphic function such that $f\neq 0$ everywhere in $\Omega$, then we define
$$
f \mathcal{D}_X:=\left(f L_z, f P_z, f h_z\right),
$$
to be scaled Enneper data of a new harmonic Enneper immersion in $\mathbb{C} \times \mathbb{R}$. We note that this surface is the harmonic Enneper graph of the harmonic function $h_1: \Omega \rightarrow \mathbb{R}$ defined by
$$
h_1(z):=h_1\left(z_0\right)+2 \operatorname{Re} \int_{z_0}^z f(z) h_z(z) d z.
$$
Also, the harmonic Enneper graph associated with $h_1$ is given by:
$$
X_1= \left(L_1+\overline{P_1}, h_1\right),
$$
where we define
$$
L_1(z):=\int_{z_0}^z f(z) L_z(z) d z, \quad P_1(z):=\int_{z_0}^z f(z) P_z(z) d z,
$$
to be well-defined holomorphic functions in $\Omega$. Moreover, we can state the following:
\begin{prop}
Let \( \mathcal{D}_X = (L_z, P_z, h_z) \) be the Enneper data of a quasiconformal harmonic Enneper immersion. Let \( f \) be a holomorphic function on a simply connected domain  \( \Omega \) such that $f\neq 0$ everywhere in $\Omega$. Then the scaled Enneper data \( f \mathcal{D}_X = (f L_z, f P_z, f h_z) \) also defines a quasiconformal harmonic Enneper immersion.
\end{prop}
\begin{remark}
In the special case where \( f \equiv i \) (the imaginary unit in \(\mathbb{C}\)) on \( \Omega \), the harmonic Enneper immersion, obtained from the Enneper data \( i \mathcal{D}_X \), is called the \emph{conjugate harmonic Enneper immersion} of \( X \).
\end{remark}

   \section{Application}\label{application}
In this section, we introduce rotational harmonic immersions and explain why the Enneper representation simplifies certain problems. As an application, we address the question of determining the number of rotational harmonic immersions passing through two coaxial circles of the same radius.
Let the annulus be defined as 
$ A=\{r_1<|z| <r_2\} \subset \mathbb{C}$, where \( r_1<r_2 \), and consider the harmonic Enneper immersion \( X: A \to \mathbb{C} \times \mathbb{R} \), given by $ X = (L + \overline{P}, h),$ 
which is assumed to be non-planar. Without loss of generality, we assume that \( X(A) \) is invariant under the family of vertical rotations
\[ B_{\theta}=\begin{bmatrix}
\cos\theta & \sin\theta & 0 \\
-\sin\theta & \cos\theta & 0 \\
0 & 0 & 1
\end{bmatrix}. \]
Let \( m \neq 0 \) denote the additive period \( i \operatorname{Im} \int_{\gamma}h_z \) of the conjugate harmonic function \( h^*: A \to \mathbb{R} \) of \( h \), where \( \gamma \) is a loop in \( A \). Define \( z = e^{(h+ih^*)\frac{2\pi}{m}} \) and observe that \( h_z(z) = \frac{dz}{z} \) and \( h(z) = \ln|z| \). Up to a symmetry with respect to a vertical plane, we can assume that \( B_\theta \) induces the biholomorphism \( z \to e^{i\theta}z \) on \( A \). Any harmonic function $k$, defined on an annulus \( A \) can be represented as
\begin{equation}
    k(z)=\sum_{-\infty}^{\infty} a_n z^n+ \frac{b_n}{\bar{z}^n}+c\ln{|z|^2}.
\end{equation}
Observe that, to satisfy \( k(e^{i\theta}z)=e^{i\theta}k(z) \), we must have
\begin{equation}
    k(z)= {a_{1}}{z}+ \frac{b_{1}}{\bar{z}}.
\end{equation}
where \( a_1, b_1 \in \mathbb{C} \).

\begin{definition}
Let \( \gamma \in \mathcal{C}^{\infty}(I) \) be a smooth function defined on an open interval \( I \subset \mathbb{R} \). A rotational surface in \( \mathbb{C} \times \mathbb{R} \cong \mathbb{R}^3(x,y,t) \) is locally parametrized by
\[
X(r, \theta) = \left(e^{i\theta} \gamma(r), r\right),
\]
where \( (r, \theta) \) are polar coordinates on $\mathbb{C}$ and
 \( \gamma(r) \) is called the profile curve (or radial function) of the surface. The radial function \( \gamma(r) \) describes how the distance from the rotational axis (i.e., the t-axis) changes as a function of height.
\end{definition}
Now, we can conclude the following:
\begin{prop}
    Every rotational harmonic Enneper immersion in \( \mathbb{C} \times \mathbb{R} \) is given by 
    \[ X(r,\theta)=\frac{1}{c}\left(e^{i\theta}\Big(ar+\frac{b}{r}\Big),\ln{r}\right), \]
    where \( a, b \in \mathbb{C} \), \( c \in \mathbb{R}^{+}\backslash\{0\} \), and \( (r,\theta) \) are the polar coordinates on \( A \).
\end{prop}

\begin{Note}
\begin{enumerate}
   \item For the fixed value \( a = b = \frac{1}{2} \) and $c=1$, we obtain the well-known minimal surface of revolution, the catenoid (see in \cite{Iwaniec2012}).
   \item  The profile curve of a rotational harmonic Enneper immersion can be expressed as \(\left( ae^cR+\frac{b}{e^cR}\right) \), where \( R \) is a parameter along the axis of rotation. Using the coordinate transformation in \( \mathbb{R}^3 \) given by \( (x,y,t)\mapsto (x,y,\ln{r}=cR) \),  we align the axis of rotation with the \( t \)-axis. This transformation simplifies the framing of the subsequent theorem.
    \end{enumerate}
\end{Note}

We now investigate the number of rotational harmonic Enneper immersions connecting two coaxial circles. The radial function for the rotational harmonic Enneper immersion is given by
\begin{equation} \label{profile curve}
    f(z) = \frac{1}{c}\left(ae^{cR}+\frac{b}{e^{cR}}\right),
\end{equation}
where \( a,b \in \mathbb{C} \) and \(c \in \mathbb{R}^{+}\backslash \{0\} \). 
\begin{theorem}
Let \( C_{\pm l} = \{ (x, y, t) \in \mathbb{R}^3 : x^2 + y^2 = r^2, \, t = \pm l \} \) be two circles of radius \( r > 0 \), centered along the \( z \)-axis and separated by a vertical distance \( 2l > 0 \). Then there exists a constant \( c_1 \in \mathbb{R}\) such that the number of rotational harmonic Enneper immersions connecting the circles \( C_{-l} \cup C_l \) are exactly two, if \( \frac{l}{r} < c_1 \); exactly one, if \( \frac{l}{r} = c_1 \) and none, if \( \frac{l}{r} > c_1 \).

\end{theorem}

\begin{proof}
   To prove this, we restate the problem in terms of planar curves in the \( xy \)-plane and observe that the only possible case arises when \( a = b \in \mathbb{R}\) in \eqref{profile curve}. The problem reduces to determining how many curves of the form
    \[ f(l) = \frac{1}{c}(ae^{cl}+\frac{a}{e^{cl}}), \]
    pass through the point \( P=(l, r) \). This is equivalent to solving the equation
    \[ \frac{a}{c}(e^{cl}+\frac{1}{e^{cl}}) = r, \]
    where \( r \) and \( a \) are given, and \( c \neq 0 \) is the unknown.
    Define the function
    \[ g(c) = \frac{a}{c}(e^{cl}+\frac{1}{e^{cl}}).\]
    We observe that
    \[ \lim _{c \to \infty} g(c) = \infty, \quad \text{or} \quad \lim _{c \to \infty} g(c) = -\infty,\]
    depending on the $a$. Without loss of generality, we assume $ \lim _{c \to \infty} g(c) = \infty$.
    Now, the derivative
    $ g^{\prime}(c)={2a (cl \sinh (c l)-\cosh (c l))}/{c^2} $
    has an unique minimum at \( \coth(cl) = cl \), yielding \( c=\frac{1.1997}{l} \).
Thus, the existence condition is given by
    \[ \frac{2al}{1.1997}\cosh(1.1997) \leq r, \]
    which simplifies to
    $ {2l}/{r} \leq {0.6627}/{a}.$
    Hence, there exists a critical value \( c_1 \approx \frac{0.3314}{a} \) such that, we have the following cases:
    \begin{enumerate}
        \item If \( l/r < c_1 \), there are exactly two solutions.
    \item  If \( l/r = c_1 \), there is exactly one solution.
   \item  If \( l/r > c_1 \), no solution exists. 
   \end{enumerate}
    Hence the claim.
\end{proof}
\section{Acknowledgement}
The author gratefully acknowledges Dr. Rahul Kumar Singh and Mr. Subham Paul for their valuable comments. The author also acknowledges the financial support from the University Grants Commission of India (UGC) under the UGC-JRF scheme (Beneficiary Code: BININ04008604).

%\section{Future work}

%The author is exploring similar results for Enneper-type representations of spacelike and timelike harmonic Enneper immersions in the Lorentz–Minkowski space $\mathbb{L}^3:= (\mathbb{R}^3,dx^2+dy^2-dz^2)$, investigating their applications based on their causal characters.

%\section{Declarations}
%\begin{itemize}
%\item  \textbf{Funding}: Not applicable.
%\item  \textbf{Data Availability}: Not applicable
%\item  \textbf{Conflict of interest}: There are no conflicts of interest to declare.
%\end{itemize}

\bibliography{Priyank}
\bibliographystyle{ieeetr}
\end{document}